\documentclass[12pt]{amsart}  

\usepackage[latin1]{inputenc}
\usepackage{amsmath, bm, mathdots} 
\usepackage{amsfonts}
\usepackage{amssymb}
\usepackage{stmaryrd}
\usepackage{latexsym} 
\usepackage{graphicx}
\usepackage{subfigure}
\usepackage{xcolor}
\usepackage{hyperref}
\usepackage{verbatim}
\usepackage[all]{xy}
\usepackage{graphics}
\usepackage{pdfsync}
\usepackage{tikz}
\usepackage{url}
\usepackage{enumerate}
\usepackage{mathtools}
\usepackage{stmaryrd}

\theoremstyle{definition}
\newtheorem{theorem}{Theorem}[section]
\newtheorem*{theorem*}{Theorem}
\newtheorem{proposition}[theorem]{Proposition}

\newtheorem{lemma}[theorem]{Lemma}
\newtheorem{corollary}[theorem]{Corollary}
\newtheorem{conjecture}[theorem]{Conjecture}

\newtheorem{example}[theorem]{Example}

\theoremstyle{remark}

\numberwithin{equation}{section}
\newcommand{\bbQ}{\mathbb{Q}}

\newcommand{\suchthat}{\,:\,}
\newcommand{\spam}{\operatorname{span}}

\newcommand{\Sym}{\ensuremath{\operatorname{Sym}}}

\newcommand{\NSym}{\ensuremath{\operatorname{NSym}}}

\newcommand{\NCSym}{\ensuremath{\operatorname{NCSym}}}
\newcommand{\slashp}{\mid}

\newcommand{\bx}{\text{\bf x}}

\newcommand{\id}{\mathrm{id}}

\newcommand{\enc}{{e}}         	
\newcommand{\mnc}{{m}}           	
\newcommand{\pnc}{{p}}                
\newcommand{\xnc}{{x}}                 

\newcommand{\mhm}{\mathbf{m}}           	
\newcommand{\phm}{\mathbf{p}}                
\newcommand{\xhm}{\mathbf{x}}  

\newcommand{\calC}{{\mathcal{C}}} 
\newcommand{\calD}{{\mathcal{D}}} 
\newcommand{\Set}{{\mathrm{Set}}} 
\newcommand{\Vect}{{\mathrm{Vect}}} 
\newcommand{\bP}{{\mathbf{P}}} 
\newcommand{\bQ}{{\mathbf{Q}}} 
\newcommand{\bPi}{{\mathbf{\Pi}}} 
\newcommand{\bone}{{\mathbf{1}}} 
 
\newcommand{\calK}{{\mathcal{K}}} 
\newcommand{\Sp}{{\mathrm{Sp}}} 
 
\newcommand{\St}{{\mathrm{St}}} 
\newcommand{\stn}{{\llbracket n \rrbracket}} 
\newcommand{\mymu}{{u}}

\newcommand\svw[1]{{\color{black}{#1}}}

\newlength\cellsize 
\savebox2{%
\begin{picture}(15,15)
\put(0,0){\line(1,0){15}}
\put(0,0){\line(0,1){15}}
\put(15,0){\line(0,1){15}}
\put(0,15){\line(1,0){15}}
\end{picture}}
\newcommand\cellify[1]{\def\thearg{#1}\def\nothing{}%
\ifx\thearg\nothing
\vrule width0pt height\cellsize depth0pt\else
\hbox to 0pt{\usebox2\hss}\fi%
\vbox to 15\unitlength{
\vss
\hbox to 15\unitlength{\hss$#1$\hss}
\vss}}
\newcommand\tableau[1]{\vtop{\let\\=\cr
\setlength\baselineskip{-16000pt}
\setlength\lineskiplimit{16000pt}
\setlength\lineskip{0pt}
\halign{&\cellify{##}\cr#1\crcr}}}
\savebox3{%
\begin{picture}(15,15)
\put(0,0){\line(1,0){15}}
\put(0,0){\line(0,1){15}}
\put(15,0){\line(0,1){15}}
\put(0,15){\line(1,0){15}}
\end{picture}}
\newcommand\expath[1]{%
\hbox to 0pt{\usebox3\hss}%
\vbox to 15\unitlength{
\vss
\hbox to 15\unitlength{\hss$#1$\hss}
\vss}}
\newcommand\bas[1]{\omit \vbox to \cellsize{ \vss \hbox to \cellsize{\hss$#1$\hss} \vss}}

\everymath{\displaystyle}
\begin{document}

\title[The extra basis in noncommuting variables]{The extra basis in noncommuting variables}

\author{Farid Aliniaeifard}
\address{
 Department of Mathematics,
 University of British Columbia,
 Vancouver BC V6T 1Z2, Canada}
\email{farid@math.ubc.ca}

\author{Stephanie van Willigenburg}
\address{
 Department of Mathematics,
 University of British Columbia,
 Vancouver BC V6T 1Z2, Canada}
\email{steph@math.ubc.ca}

\thanks{\svw{Both} authors were supported in part by the \svw{Natural} Sciences and Engineering Research Council of Canada.}
\subjclass[2010]{05E05, 05E10, 18M80, 20C30}
\keywords{\svw{coproduct,} Hopf monoid,   noncommutative symmetric function, symmetric function}

\begin{abstract} We answer a question of Bergeron, Hohlweg, Rosas, and Zabrocki from 2006 to give a combinatorial description for the coproduct of the $\xnc$-basis in the Hopf algebra of symmetric functions in noncommuting variables, $\NCSym$, which arises in the theory of Grothendieck bialgebras. We achieve this using the theory of Hopf monoids and the Fock functor. We also determine combinatorial expansions of this basis in terms of the monomial and power sum symmetric functions in $\NCSym$, and by taking the commutative image of the $\xnc$-basis we discover a new multiplicative basis for the algebra of symmetric functions.\end{abstract}

\maketitle
%

\section{Introduction}\label{sec:intro}  
Symmetric functions in noncommuting variables, $\NCSym$, were introduced in 1936 by Wolf \cite{Wolf} who produced a noncommutative analogue of the fundamental theorem of symmetric functions. Her work was generalized in 1969 by Bergman and Cohn \cite{BC}, but there were few further discoveries until 2004 when Rosas and Sagan \cite{RS} substantially advanced the area, discovering numerous noncommutative analogues of classical concepts in symmetric function theory, such as the RSK map and the $\omega$ involution. In particular, they gave noncommutative analogues of the bases of monomial, power sum, elementary, and complete homogeneous symmetric functions, each analogue mapping to its (scaled) symmetric function analogue under the projection map that lets the variables commute.

{Rosas and Sagan did define {a partial set of} Schur functions in noncommuting variables, which projected to scaled classical Schur functions.  However, they were not a basis for NCSym since they were indexed by integer \svw{partitions} rather than \svw{by} set partitions.} This led them to ask in \cite[Section 9]{RS} whether there is a way to define functions indexed by set partitions having properties analogous to the classical Schur functions. In 2022, a family of Schur functions in noncommuting variables was introduced in \cite{ALvW22} and Rosas and Sagan's question was answered.  In an earlier attempt to answer Rosas and Sagan's question, in 2006, Bergeron, Hohlweg, Rosas, and Zabrocki \cite{BHRZ} suggested a possible candidate, the $\xnc$-basis, which corresponds to the classes of simple modules of the Grothedieck bialgebra of the semi-tower of partition lattice algebras. They also gave the product and internal coproduct of the $x$-basis; however, the coproduct formula for the $x$-basis remained open, and in the last sentence of their paper, they \svw{stated that}  ``It would be interesting to find a formula for $\Delta(x_\pi)$", where \svw{$\Delta$ denotes the coproduct} and $\pi$ is a set partition of $[n]=\{1,2,\dots, n\}$. The $x$-basis has been the subject of other studies; for example,
the expansion of the chromatic symmetric functions in $\NCSym$ in terms of the $x$-basis was studied in  \cite{DvW}.

In this paper,  we delve further into the study of the $x$-basis of $\NCSym$ and its analogue in the Hopf monoid of set partitions, \svw{the} $\xhm$-basis. We answer the question \svw{of} Bergeron, Hohlweg, Rosas, and Zabrocki from 2006 and find a combinatorial formula for the coproduct of the $\xnc$-basis of $\NCSym$. Also, we describe the \svw{expansion} of the $\xnc$-basis elements in terms of other bases of $\NCSym$. \svw{In particular,} for every set partition $\pi$ of $[n]$, the coefficients in the expansion of $\xnc_\pi$ in terms of the $\mnc$-basis can be described by acyclic \svw{orientations} with one sink of certain complete multipartite graphs.

More precisely, our paper is structured as follows. 
In Section \ref{sec:background}, we present the Hopf algebras of symmetric functions $\Sym$ and symmetric functions in noncommuting variables $\NCSym$, then we describe the Hopf monoid of set partitions ${\bf \Pi}$. Moreover, we show how the Hopf monoid of set partitions is related to the Hopf algebra of symmetric functions in noncommuting variables via the Fock functor. In Section \ref{sec:coproduct}, given a set partition $A$ of a finite set $S$, we obtain a coproduct formula for $\xhm_A$ in ${\bf \Pi}$ in Theorem \ref{thm:prod-coprod}. Then, in Corollary \ref{cor:coproduct}, we apply the Fock functor on $\bPi$ to find a coproduct formula for $\xnc_\pi$ in $\NCSym$ for every set partition $\pi$ of $[n]$. Also, by applying $\rho$, the algebra morphism that lets the variables commute, \svw{to} the $x$-basis elements of $\NCSym$, in Proposition \ref{prop:xbasisSym}, we introduce the $x$-basis of $\Sym$, which is a multiplicative basis. In Section \ref{sec:change}, we expand $\xnc_\pi$ in terms of other bases of $\NCSym$, \svw{in particular,} we show in Theorem \ref{thm:x-to-m} that the coefficients appearing in the expansion of $\xnc_\pi$ in terms of the $\mnc$-basis up to a sign are the sum of the numbers of acyclic orientations with one sink of certain complete multipartite graphs obtained from set partitions; so they are all negative or \svw{all} positive. Finally, we show in Theorem \ref{thm:e-pos} that $x_{\stn}$ is $\enc$-positive in $\NCSym$ if and only if $x_{n}$ is $e$-positive in $\Sym$, where $\llbracket n \rrbracket$ is the set partition of $n$ with only one block.

\section{Preliminaries}\label{sec:background}
%
%
%
%
%




\subsection{Integer and set partitions}
An \emph{integer partition} $\lambda$ of a positive integer $n$, denoted by $\lambda\vdash n$, is a nonincreasing list $\lambda_1\lambda_2\svw{\cdots} \lambda_l$ of positive integers whose sum is $n$. Each positive integer $\lambda_i$ is called a \emph{part} of $\lambda$. Let $m_i$ be the number of parts of $\lambda$ that are equal to $i$. 
 We let $\ell(\lambda)$ be the number of parts of $\lambda$, and $$ \lambda!=\prod_{i=1}^{\ell(\lambda)}\svw{\lambda_i !}\quad \text{and}\quad  \lambda^!=\prod_{i=1}^{n}m_i!.$$
 
 Given $\lambda\vdash n$ and $\gamma\vdash m$, let $\lambda|\gamma$ be the integer partition of $n+m$ whose parts are the parts of $\lambda$ and $\gamma$. For example, if $\lambda=3221$ and $\gamma=221$, then $\lambda|\gamma=3222211$.
 
 A \emph{decomposition} of a finite set $S$ is a finite sequence $(S_1,S_2, \dots, S_l)$ of disjoint subsets (possibly empty)
of $S$ whose union is $S$. In this situation, we write
$S=S_1\sqcup S_2 \sqcup \cdots \sqcup S_l$. 
  A \emph{set partition} $A$ of a finite set $S$, denoted by $A\vdash S$, is a collection of disjoint nonempty subsets $S_1,S_2, \dots, S_l$ of $S$ whose union is $S$. Each subset $S_i$ is called a \emph{block} of $A$.   We usually denote a set partition $A$ with blocks $S_1, S_2, \dots, S_l$ by $A=S_1/S_2/\cdots/S_l$.  Also, we denote by $\ell(A)$ the number of blocks of the set partition $A$ and \svw{by} $\emptyset$ the unique empty set partition of $[0]=\emptyset$. 
  
  Given a set partition $A$ of $S$ and $T\subseteq S$, let $A|_T$ be the set partition of $T$ whose blocks are the nonempty intersections of the blocks of $A$ with $T$. Let $S=S_1\sqcup  S_2$. Given set partitions $A$ of $S_1$ and $B$ of $S_2$, their \emph{disjoint union} is the set partition $A\sqcup B$ of $S$ whose blocks are the disjoint union of the blocks of $A$ and $B$.

Let $\Pi[S]$ be the set of all set partitions of $S$. We have that $\Pi[S]$ is a lattice where the ordering is by refinement, that is, for $A, B \vdash   S$, \svw{we write} $B\leq A$ if each block of $B$ is contained in some block of $A$. The meet (greatest lower bound) and join (least upper bound) operations in $\Pi[S]$ are denoted by $\wedge$
and $\vee$, respectively.
Let $\mymu$ be the M\"obius function of $\Pi[S]$, then for $B\leq A$,  by \cite[Section 3]{RS} we have 
\begin{equation}\label{eq:Mobius}
\mymu(B,A)= \prod_{i=1}^{\ell(\lambda)} (-1)^{\lambda_i(B,A)-1}(\lambda_i(B,A)-1)!,
\end{equation} 
 where $\lambda_i(B,A)$ is the number of blocks of $B$ contained in the $i$th block of $A$ for any ordering of the blocks of $A$. 
  
   For convenience, when writing a set partition of a subset of integers, we
usually omit the set parentheses and commas
of the blocks. For example,  $13/24$ is a set partition of $\{1,2,3,4\}$ with blocks $\{1,3\}$ and $\{2,4\}$. Let $[n]=\{1,2,\ldots,n\}$.  We denote by $\stn$ the set partition $12\svw{\cdots} n \vdash [n]$.    
If $\pi\vdash [n]$, by Equation \eqref{eq:Mobius},
\begin{equation}\label{eq:mymun}\mymu(\pi,\stn)= (-1)^{\ell(\pi)-1}(\ell(\pi)-1)!.\end{equation}
\begin{example}\label{ex:mu}
Let $\pi=13/24\vdash [4]$ and $\sigma=1/3/24\vdash [4]$. Then 
$$u(\svw{\sigma, \pi})=(-1)^{2-1}(2-1)!(-1)^{1-1}0!=-1\quad \text{and}\quad u(\sigma,\llbracket 4 \rrbracket)=(-1)^{3-1}(3-1)!=2.$$
\end{example}

Given a subset $S$ of integers, let $\St$ be the unique order-preserving map from $S$ to $|S|$. For $A\vdash S$, define $\St(A)\vdash [|S|]$ to be the set partition of $[|S|]$ obtained by replacing $i$ \svw{with} $\St(i)$ in $A$, and let the \emph{shape} of $A$ to be the integer partition $\lambda(A)$ whose parts are the block sizes of $A$ listed in nonincreasing order. For example, $\St(167/38)=134/25$ and $\lambda(167/38)=32$.

  Given two set partitions $\pi=S_1/S_2/\cdots/S_l\vdash [n]$ and $\sigma=T_1/T_2/\cdots/T_k\vdash [m]$, their \emph{slash product} is 
$$\pi|\sigma=S_1/S_2/\cdots/S_l/T_1+n/T_2+n/\cdots/T_k+n\vdash [n+m],$$
where $T_i+n=\{t+n:t\in T_i\}$ for each $i$.  For $\lambda=\lambda_1\lambda_2\svw{\cdots} \lambda_l$, let $\llbracket \lambda \rrbracket=\llbracket\lambda_1\rrbracket|\llbracket\lambda_2\rrbracket|\cdots|\llbracket\lambda_l\rrbracket.$

Given any set partition $\pi\vdash [n]$ and $\eta\in \mathfrak{S}_n$, where $\mathfrak{S}_n$ is the
symmetric group of size $n!$, define $\eta(\pi)$ to be the set partition obtained by replacing $i$ \svw{with} $\eta(i)$ in $\pi$. For example, if $\pi=13/24$ and $\eta=1324$, \svw{in one-line notation,} then $\eta(\pi)=\eta(1)\eta(3)/\eta(2)\eta(4)=12/34$.   For any set partition $\pi\vdash [n]$ of shape $\lambda$, we can find a permutation $\eta\in \mathfrak{S}_n$ such that  $\eta(\llbracket \lambda\rrbracket)=\pi$. For example, if $\pi=13/24$, then $\lambda(\pi)=22$, and  when $\eta=1324$, 
$$\eta(12|12)=\eta(12/34)=\eta(1)
\eta(2)/\eta(3)\eta(4)=13/24.$$

\subsection{Symmetric functions}
Let $\mathbb{Q}[[x_1,x_2,\dots]]$ be the algebra of formal power series in commuting variables $\{x_1,x_2,\dots\}$.
The \emph{algebra of symmetric functions}, {$\Sym$,} is
$$\Sym = \Sym ^0 \oplus \Sym ^1 \oplus \cdots \subset \bbQ [[x_1, x_2, \ldots ]]$$ where $\Sym ^0 = \spam \{1\}$ and 
{the} $n$th graded piece for $n\geq 1$ has the following bases, known respectively as the $n$th graded piece of the \emph{$\mnc$-, $\pnc$-, $\enc$-basis of $\Sym$},
{$$\begin{array}{rclclcc}
\Sym ^n&=&\mathbb{Q}\text{-span}\{ m_\lambda \suchthat \lambda\vdash n\} &=&\mathbb{Q}\text{-span}\{ p_\lambda \suchthat \lambda\vdash n\}&=& \mathbb{Q}\text{-span}\{ e_\lambda \suchthat \lambda\vdash n\} 
\end{array}$$}\svw{where} these functions are defined as follows, {given an integer partition} $\lambda = \lambda_1\lambda_2\cdots \lambda_{l}\vdash n$.

The \emph{monomial symmetric function}, $m_\lambda$, is given by
$$m_\lambda = \sum x_{i_1}^{\lambda _1}x_{i_2}^{\lambda _2}\cdots x_{i_{l}}^{\lambda _{l}}$$summed over distinct monomials and the $i_j$ are also distinct.

\begin{example}\label{ex:monomials}
$m_{21} = x_1^2x_2 + x_2^2x_1+ \cdots$
\end{example}

The \emph{power sum symmetric function}, $p_\lambda$, is given by
$$p_\lambda = p_{\lambda _1}p_{\lambda _2}\cdots p_{\lambda _{l}}$$where {$p_{i} = x_1^{i}+x_2^{i}+\cdots$.}

\begin{example}\label{ex:powers}
$p_{21} = (x_1^2+x_2^2+ \cdots)(x_1+x_2+\cdots)$
\end{example}

The \emph{elementary symmetric function}, $e_\lambda$, is given by
$$e_\lambda = e_{\lambda _1}e_{\lambda _2}\cdots e_{\lambda _{l}}$$where {$e_{i} = \sum_{j_1<j_2<\cdots <j_{i}} x_{j_1}x_{j_2}\cdots x_{j_{i}}$.}

\begin{example}\label{ex:elementaries}
$e_{21} = (x_1x_2+x_2x_3+ \cdots)(x_1+x_2+\cdots)$
\end{example}

\subsection{Symmetric functions in noncommuting variables}\label{sec:NCSym}
Let $\mathbb{Q} [[ \bx_1, \bx_2, \ldots ]]$ be the algebra of formal power series in noncommuting variables $\{\bx_1,\bx_2,\dots\}$.
 The \svw{\emph{ algebra of symmetric functions in noncommuting variables},} $\NCSym$, is
$$\NCSym  = \NCSym ^0 \oplus \NCSym ^1 \oplus \cdots \subset \mathbb{Q} [[ \bx_1, \bx_2, \ldots ]]$$ where $\NCSym ^0 = \spam \{1\}$ and {the} $n$th graded piece for $n\geq 1$ has the following bases {\cite{RS}}, known respectively as the $n$th graded piece of the \emph{$\mnc$-, $\pnc$-, $\enc$-basis of $\NCSym$},
{$$\begin{array}{rclclcc}
\NCSym ^n&=& \mathbb{Q}\text{-span}\{ \mnc_\pi \suchthat \pi\vdash [n]\} &=&\mathbb{Q}\text{-span}\{ \pnc_\pi \suchthat \pi\vdash [n]\}
&=& \mathbb{Q}\text{-span}\{ \enc_\pi \suchthat \pi\vdash [n]\} 
\end{array}$$}where these functions are defined as follows, given a set partition $\pi \vdash [n]$.

The \emph{monomial symmetric function in $\NCSym$}, $\mnc_\pi$, is given by 
$$\mnc_\pi = \sum _{(i_1, i_2, \ldots , i_n)} \bx_{i_1}\bx_{i_2} \cdots \bx_{i_n}$$summed over all tuples $(i_1, i_2, \ldots , i_n)$ with $i_j=i_k$ if and only if $j$ and $k$ are in the same block of $\pi$.

\begin{example}\label{ex:mpi}
$\mnc_{13/2}=\bx_1\bx_2\bx_1+\bx_2\bx_1\bx_2+\bx_1\bx_3\bx_1+\bx_3\bx_1\bx_3+\bx_2\bx_3\bx_2+\bx_3\bx_2\bx_3+\cdots$
\end{example}

The \emph{power sum symmetric function in $\NCSym$}, $\pnc_\pi$, is given by 
$$\pnc_\pi = \sum _{(i_1, i_2, \ldots , i_n)} \bx_{i_1}\bx_{i_2} \cdots \bx_{i_n}$$summed over all tuples $(i_1, i_2, \ldots , i_n)$ with $i_j=i_k$ if  $j$ and $k$ are in the same block of $\pi$.

\begin{example}\label{ex:ppi}
$\pnc_{13/2}=\bx_1\bx_2\bx_1+\bx_2\bx_1\bx_2+\cdots + \bx_1^3+\bx_2^3 +\cdots $
\end{example}

The \emph{elementary symmetric function in $\NCSym$}, $\enc_\pi$, is given by 
$$\enc_\pi = \sum _{(i_1, i_2, \ldots , i_n)} \bx_{i_1}\bx_{i_2} \cdots \bx_{i_n}$$summed over all tuples $(i_1, i_2, \ldots , i_n)$ with $i_j\neq i_k$ if  $j$ and $k$ are in the same block of $\pi$.

\begin{example}\label{ex:epi}
$\enc_{13/2}= {\bx_1\bx_1\bx_2+\bx_1\bx_2\bx_2+\bx_2\bx_2\bx_1+\bx_2\bx_1\bx_1}+\cdots + \bx_1\bx_2\bx_3+\bx_2\bx_3\bx_4 +\cdots$
\end{example}

\svw{The next function will be the main object of study in this paper, and was introduced in \cite{BHRZ}. While it was not given a name in that paper, due to the notation chosen at that time, we give it the following name.} 

The \emph{extra symmetric function in $\NCSym$}, $\xnc_\pi$, is given by 
\begin{equation}\label{eq:p-basis} \xnc_\pi=\sum_{\sigma\leq \pi} \mymu(\svw{\sigma, \pi}) \pnc_\sigma.\end{equation}
\begin{example}\label{ex:xpi}
$\xnc_{13/2}=\mymu(13/2,13/2)\pnc_{13/2}+\mymu(1/2/3,13/2)\pnc_{1/2/3}=\pnc_{13/2}-\pnc_{1/2/3}=\bx_1\bx_2\bx_1+\bx_2\bx_1\bx_2+\cdots + \bx_1^3+\bx_2^3 +\cdots - (\bx_1\bx_2\bx_1+\bx_2\bx_1\bx_2+\cdots + \bx_1^3+\bx_2^3 +\cdots+\bx_1\bx_2\bx_3+\bx_1\bx_3\bx_2+\bx_2\bx_1\bx_3+\bx_2\bx_3\bx_1+\bx_3\bx_1\bx_2+\bx_3\bx_2\bx_1+\cdots+\bx_1^2\bx_2+\bx_2^2\bx_1+\cdots+\bx_1\bx_2^2+\bx_2\bx_1^2+\cdots)=-(\bx_1\bx_2\bx_3+\bx_1\bx_3\bx_2+\bx_2\bx_1\bx_3+\bx_2\bx_3\bx_1+\bx_3\bx_1\bx_2+\bx_3\bx_2\bx_1+\cdots)-(\bx_1^2\bx_2+\bx_2^2\bx_1+\cdots)-(\bx_1\bx_2^2+\bx_2\bx_1^2+\cdots)=-m_{1/2/3}-m_{12/3}-m_{1/23}$
\end{example}
We have \svw{from \cite{BHRZ}}  that $$\NCSym^n=\mathbb{Q}\text{-span}\{ \xnc_\pi \suchthat \pi\vdash [n]\},$$ and the set $\{x_\pi: \pi \text{~is a set partition}\}$ is called the \emph{$x$-basis} of $\NCSym$. 

The algebra $\NCSym$ is indeed a Hopf algebra, and the product \svw{$\mu$} and coproduct \svw{$\Delta$} of the $p$-basis are given as follows \cite{BHRZ}. Given $\pi\vdash [n]$ and $\sigma\vdash [m]$,  we have
\begin{equation}\label{eq:prodNCSym} 
\mu(p_\pi\otimes p_\sigma)=p_{\pi|\sigma}.
\end{equation} 
 Also  for $\pi\vdash [n]$, 
\begin{equation}\label{eq:coprodNCSym}
\Delta(p_\pi)=
\sum_{[n]=S_1\sqcup S_2: \atop{\pi=\pi|_{S_1}\sqcup \pi|_{S_2}}}p_{\St(\pi|_{S_1})}\otimes p_{\St(\pi|_{S_2})}.
\end{equation} 

\begin{example}
We have that 
$$\mu(p_{13/24}\otimes p_{13/2})=p_{13/24/57/6}$$ and 
$$\Delta(p_{13/2})=1\otimes p_{13/2} +p_{12}\otimes p_{1}+p_{1}\otimes p_{12}+p_{13/2}\otimes 1.$$
\end{example}

Let $\rho$ be the algebra morphism
$$\rho \suchthat \bbQ [[ \bx_1, \bx_2, \ldots ]] \rightarrow \bbQ [[ x_1, x_2, \ldots ]]$$ that simply lets the variables commute. We have the following result.

\begin{lemma}\label{lem:RSrho} \cite[Theorem 2.1]{RS} Let $\pi$ be a set partition of $[n]$. Then
\begin{enumerate}[1)]
\item $\rho(m_\pi) = \lambda (\pi) ^! m_{\lambda(\pi)}$,
\item $\rho(p_\pi) = p_{\lambda(\pi)}$, and
\item $\rho(e_\pi) = \lambda (\pi)! e_{\lambda(\pi)}$.
\end{enumerate}
\end{lemma}

The classical involution $\omega$
defined on $\Sym$, whose analogue in $\NCSym$ is also an involution denoted by $\omega$ \cite[p. 221]{RS}, \svw{is}
defined by
\begin{equation}\label{eq:omega}
\omega(p_\lambda)=(-1)^{n-\ell(\lambda)}p_\lambda \quad \quad \omega(\pnc_\pi)=(-1)^{n-\ell(\pi)}\pnc_\pi,
\end{equation}
for every $\lambda\vdash n$ and $\pi\vdash [n]$.

Given a permutation $\eta\in\mathfrak{S}_n$ and any  monomial $\bx_{i_1}\bx_{i_2}\cdots \bx_{i_n}$ in noncommuting variables, define 
$$\eta\circ (\bx_{i_1}\bx_{i_2}\cdots \bx_{i_n})=\bx_{i_{\eta^{-1}(1)}}\bx_{i_{\eta^{-1}(2)}}\cdots \bx_{i_{\eta^{-1}(n)}},$$ and extend it linearly. Then, define the linear transformation
$$\begin{array}{cccc}
R:& \bbQ [[x_1, x_2, \ldots ]] & \rightarrow& \mathbb{Q} [[ \bx_1, \bx_2, \ldots ]]  \\
& x_{i_1}x_{i_2}\cdots x_{i_n} & \mapsto & \frac{1}{n!}\sum_{\eta\in\mathfrak{S}_n} \eta\circ (\bx_{i_1}\bx_{i_2}\cdots \bx_{i_n}).
\end{array}$$
Note that $\rho R$ is \svw{the} identity \svw{map.} 
Also, by \cite[Equation (4.2)]{DvW} and  \cite[p. 219]{RS}, if $b_\pi$ is any basis element of any of the
above bases of $\NCSym$ and $\eta$ is a permutation, then

\begin{equation}\label{eq:actbasis}
\eta\circ b_\pi=b_{\eta(\pi)}.
\end{equation}

\subsection{Hopf monoids} 
A \emph{vector species} is a functor from the category of finite sets with bijections as morphisms, $\Set^\times$, to the category of vector spaces over $\mathbb{Q}$ with linear transformations as morphisms, $\Vect_\bbQ$. Thus, any vector species $\bP$ takes a finite set $I$ and gives a vector space $\bP[I]$. Moreover, the value of the vector species $\bP$ at the bijection $f:S\rightarrow T$ is denoted by $$\bP[f]=\bP[S]\rightarrow \bP[T].$$
For complete details about Hopf monoids, see \cite{AM}.

\begin{example} \label{ex:MonSet}
 The {\it vector species of set partitions} is  $\bPi: \Set^\times\rightarrow\Vect_\bbQ,$ where 
$$\bPi[S]=\bbQ\text{-span}\{\mhm_{A}: A \vdash S\},$$
and  {if $f: S\rightarrow T$ is a bijection, then 
$$\bPi[f]:\bPi[S]\rightarrow \bPi[T]$$ is defined by the linear extension of the map,
$$\bPi[f](\mhm_{A})=\mhm_{f(A)},$$ where $f(A)$ is the set partition of $T$ obtained by replacing $i$ \svw{with} $f(i)$ in $A$. }
 For example, $\bPi[\St](\mhm_{16/38})=\mhm_{\St(16/38)}=\mhm_{13/24}$.
 \end{example} 

%

\begin{example} 
Let $V$ be a vector space. Define the vector species $$\bone_V: \Set^\times \rightarrow \Vect_\bbQ,$$ where $$\bone_V[S]=\begin{cases} V &  S=\emptyset,\\ 0& \text{otherwise.} \end{cases} $$
\end{example}

{
A \emph{natural transformation} $\alpha$ between functors $F, G:\calC\rightarrow \calD$, denoted by $\alpha:F\rightarrow G$, is an assignment that assigns to each object $I\in \calC$ a morphism $\alpha_I:F[I]\rightarrow G[I]$ of $\calD$ such that for any morphism $f:I\rightarrow J$ of $\calC$ the following diagram 
$$
\begin{tikzpicture} 
\node(FC) at (0,0){$F[I]$};
\node(GC) at (2,0){$G[I]$};
\node(FD) at (0,-2){$F[J]$};
\node(GD) at (2,-2){$G[J]$};
\draw[->] (FC)--(GC);
\node at (1,0.2){$\alpha_I$};
\draw[->] (FD)--(GD);
\node at (1,-1.8){$\alpha_J$};
\draw[->] (FC)--(FD);
\node at (-0.5,-1){$F[f]$};
\draw[->] (GC)--(GD);
\node at (2.5,-1){$G[f]$};
\end{tikzpicture} 
 $$
commutes.
}

The \emph{Cauchy product} of the vector species $\bP$ and $\bQ$ is the vector species 
$$\bP \cdot \bQ:\Set^\times \rightarrow \Vect_\bbQ,$$ where
$$(\bP \cdot \bQ)[S]=\bigoplus_{S=S_1\sqcup S_2}\bP[S_1]\otimes \bQ[S_2],$$
and if $f:S\rightarrow T$ is a morphism, then 
$$
\begin{array}{cccc}
(\bP \cdot \bQ)[f]:&\bigoplus_{S=S_1\sqcup S_2}\bP[S_1]\otimes \bQ[S_2]& \rightarrow &\bigoplus_{T=T_1\sqcup T_2}\bP[T_1]\otimes \bQ[T_2]
\end{array}$$ is the direct sum of the maps 
$$\bP[S_1]\otimes \bQ[S_2] \xrightarrow{\bP[f|_{S_1}]\otimes \bQ[f|_{S_2}] } \bP[f(S_1)]\otimes \bQ[f(S_2)]$$ 
over all pairs $(S_1,S_2)$ with $S=S_1\sqcup S_2$.
%
%
%

A \emph{monoid} is a vector species $\bP$ together with morphisms of species 
$$\mu:\bP\cdot \bP \rightarrow \bP\quad \text{and} \quad \iota: \bone_\bbQ  \rightarrow \bP$$   
subject to the following axioms. 

\emph{Naturality axiom}.  Let $S=S_1\sqcup S_2$. For each bijection $f:S\rightarrow T$, the diagram 
$$
\begin{tikzpicture} 
\node(a) at (1,0){$\bP[S_1]\otimes \bP[S_2]$};
\node(b)  at (8,0){$\bP[S]$};
\node(c) at (1,-2){$\bP[f(S_1)]\otimes \bP[f(S_2)]$};
\node(d)  at (8,-2){$\bP[T]$};

\draw[->] (a)--(b);
\node at (5,0.3){$\mu_{S_1,S_2}$};
\draw[->] (c)--(d);
\node at (5,-1.7){$ \mu_{f(S_1),f(S_2)}$};
\draw[->] (a)--(c);
\node at (8.5,-1){$\bP[f]$};
\draw[->] (b)--(d);
\node at (-0.6,-1){$\bP[f|_{S_1}] \otimes \bP[f|_{S_2}]$};
\end{tikzpicture} 
$$
commutes.

\emph{Associativity axiom}. For each set partition $S=S_1\sqcup S_2\sqcup S_3$, the diagram 
$$
\begin{tikzpicture} 
\node(a) at (0,0){$\bP[S_1]\otimes \bP[S_2]\otimes \bP[S_3]$};
\node(b)  at (8,0){$\bP[S_1]\otimes \bP[S_2\sqcup S_3]$};
\node(c) at (0,-2){$\bP[S_1\sqcup S_2]\otimes \bP[S_3]$};
\node(d)  at (8,-2){$\bP[S]$};

\draw[->] (a)--(b);
\node at (4,0.3){$\id_{S_1}\otimes \mu_{S_2,S_3}$};
\draw[->] (c)--(d);
\node at (4,-1.7){$ \mu_{S_1\sqcup S_2,S_3}$};
\draw[->] (a)--(c);
\node at (8.9,-1){$\mu_{S_1,S_2\sqcup S_3}$};
\draw[->] (b)--(d);
\node at (-1.1,-1){$\mu_{S_1,S_2} \otimes  \id_{S_3}$};
\end{tikzpicture} 
$$
commutes. 

\emph{Unit axiom}. For each finite set $S$, the diagrams
$$
\begin{tikzpicture} 
\node(a) at (0,0){$\bP[S]$};
\node(b)  at (5,0){$\bP[\emptyset]\otimes \bP[S]$};
\node(c)  at (5,-2){$\bbQ\otimes\bP[S]$};

\draw[->] (b)--(a);
\node at (2,0.3){$ \mu_{\emptyset,S}$};
\draw[->] (c)--(b);
\node at (5.8,-1){$\iota_{\emptyset}\otimes \id_S$};
\draw[double] (c)--(a);

\end{tikzpicture} 
\quad  \quad 
\begin{tikzpicture} 
\node(a)  at (0,0){$\bP[S]\otimes \bP[\emptyset]$};
\node(b) at (5,0){$\bP[S]$};
\node(c)  at (0,-2){$\bP[S]\otimes {\bbQ}$};

\draw[->] (a)--(b);
\node at (2.5,0.3){$ \mu_{S,\emptyset}$};
\draw[double] (c)--(b);
\node at (-0.8,-1){$\id_S\otimes \iota_{\emptyset}$};
\draw[->] (c)--(a);

\end{tikzpicture} 
$$

commute.

A  \emph{comonoid} is defined dually with
 two morphisms of species 
$$\Delta: \bP\rightarrow \bP \cdot \bP\quad \quad \text{and}\quad \quad \varepsilon: \bP \rightarrow \bone_\bbQ.$$

A \emph{Hopf monoid} $\bP$ is a monoid and a comonoid that satisfies the following axiom. 

\emph{Compatibility axiom}. Let $S_1\sqcup S_2$ and $T_1\sqcup T_2$ be set partitions of $S$, and consider the pairwise \svw{intersections} 
$$I=S_1\cap T_1, J=S_1\cap T_2, K=S_2\cap T_1, L=S_2\cap T_2.$$  Then the diagram 
$$
\begin{tikzpicture} 
\node(a) at (0,0){$\bP[I]\otimes \bP[J]\otimes \bP[K]\otimes \bP[L]$};
\node(b) at (10,0){$\bP[I]\otimes \bP[K]\otimes \bP[J]\otimes \bP[L]$};
\node(c) at (0,-3){$\bP[S_1]\otimes \bP[S_2]$};
\node(d) at (5,-3){$\bP[S]$};
\node(e) at (10,-3){$\bP[T_1]\otimes \bP[T_2]$};

\draw[->,thick](a)--(b);
\node at (5,0.3){$\id_I\otimes {\rm twist} \otimes \id_L$};
\draw[->,thick](b)--(e);
\node at (11.2,-1.5){$\mu_{I,K}\otimes \mu_{J,L}$};
\draw[->,thick](c)--(a);
\node at (-1.2,-1.5){$\Delta_{I,J}\otimes \Delta_{K,L}$};
\node at (3,-3.3){$\mu_{S_1,S_2}$};
\draw[->,thick](c)--(d);
\node at (7,-3.3){$\Delta_{T_1,T_2}$};
\draw[->,thick](d)--(e);

\end{tikzpicture} 
$$
commutes, where the map ${\rm twist}$ maps $a\otimes b$ to $b\otimes a$. In addition, \svw{the} diagrams
$$
\begin{tikzpicture} 
\node(a) at (0,0){$\bP[\emptyset]\otimes \bP[\emptyset]$};
\node(b) at (5,0){$\bbQ\otimes \bbQ$};
\node(c) at (0,-2){$\bP[\emptyset]$};
\node(d) at (5,-2){$\bbQ$};
\draw[->] (a)--(b);
\node at (2.5,0.3){$\varepsilon_\emptyset\otimes \varepsilon_\emptyset$};
\draw[->] (c)--(d);
\node at (2.5,-2.3){$\varepsilon_\emptyset$};
\draw[->] (a)--(c);
\node at (-0.5,-1){$\mu_{\emptyset,\emptyset}$};
\draw[double] (b)--(d);
\node at (4.5,-1){};
\end{tikzpicture} 
\quad \quad 
\begin{tikzpicture} 
\node(a) at (0,0){$\bP[\emptyset]\otimes \bP[\emptyset]$};
\node(b) at (5,0){$\bbQ\otimes \bbQ$};
\node(c) at (0,-2){$\bP[\emptyset]$};
\node(d) at (5,-2){$\bbQ$};
\draw[->] (b)--(a);
\node at (2.5,0.3){$\iota_\emptyset\otimes \iota_\emptyset$};
\draw[->] (d)--(c);
\node at (2.5,-2.3){$\iota_\emptyset$};
\draw[->] (c)--(a);
\node at (-0.5,-1){$\Delta_{\emptyset,\emptyset}$};
\draw[double] (d)--(b);
\node at (4.5,-1){};
\end{tikzpicture} 
$$
$$
\begin{tikzpicture}
\node(a) at (0,0){$\bP[\emptyset]$};
\node(b) at (-2,-2){$\bbQ$};
\node(c) at (2,-2){$\bbQ$};

\draw[->](b)--(a);
\node at (-1.2,-0.8){$\iota_{\emptyset}$};
\draw[->](a)--(c);
\node at (1.2,-0.8){$\varepsilon_{\emptyset}$};
\draw[double](b)--(c);
\end{tikzpicture} 
$$
commute.

\subsection{The Hopf monoid of set partitions}
The vector species ${\bf \Pi}$ introduced in Example $\ref{ex:MonSet}$ is a Hopf monoid. One of its bases is \svw{the} $\mhm$-basis, $\{\mhm_A\}$, and the product and coproduct formulas for $\{\mhm_A\}$ are as follows.

Let $S=S_1\sqcup S_2$. \svw{Then the} product of ${\bf \Pi}$ is
$$
\begin{array}{cccc}
\mu_{S_1,S_2}:&  {\bf \Pi}[S_1] \otimes {\bf \Pi}[S_2] & \rightarrow & {\bf \Pi}[S]
\end{array} 
$$
 given by 
$$\mu_{S_1,S_2}(\mhm_{A} \otimes \mhm_{B})=\sum_{C\vdash S: \atop{ C|_{S_1}=A, C|_{S_2}=B}} \mhm_C.$$
Also, its coproduct is
$$
\begin{array}{cccc}
\Delta_{S_1,S_2}:& {\bf \Pi}[S]& \rightarrow  & {\bf \Pi}[S_1] \otimes {\bf \Pi}[S_2] 
\end{array} 
$$
given by 
$$\Delta_{S_1,S_2}(\mhm_A)=\begin{cases}
\mhm_{A|_{S_1}} \otimes \mhm_{A|_{S_2}} & A=A|_{S_1}\sqcup A|_{S_2},\\
0 & \text{otherwise.}
\end{cases}$$
\begin{example}
We have that 
$$\mu_{\{1,2\},\{3,4,5\}}(\mhm_{12} \otimes \mhm_{35/4})=\mhm_{12/35/4}+\mhm_{1235/4}+\mhm_{124/35}$$
and 
$$\Delta_{\{1,2\},\{3,4,5\}}(\mhm_{12/35/4})=\mhm_{12}\otimes \mhm_{35/4}.$$
\end{example}

Two other bases of $\bPi$ are the $\phm$-basis, $\{\phm_A\}$, and the $\xhm$-basis, $\{\xhm_{A}\}$, defined by 
$$\phm_A=\sum_{B\geq A} \mhm_B\quad\quad\text{~and~}\quad \quad \xhm_A=\sum_{B\leq A} \mymu(B,A) \phm_B,$$ respectively. Note that $\phm_A=\sum_{B\leq A} \xhm_B.$

By \cite[Section 1.3]{BB}, the product and coproduct for \svw{the} $\phm$-basis are
$$\mu_{S_1,S_2}(\phm_A\otimes \phm_B)=\phm_{A\sqcup B}$$ and 
\begin{equation}\label{eq:cop-pi}
\Delta_{S_1,S_2}(\phm_A)=\begin{cases}
\phm_{A|_{S_1}} \otimes \phm_{A|_{S_2}} & A=A|_{S_1}\sqcup A|_{S_2},\\
0 & \text{otherwise.}
\end{cases} 
\end{equation} 
\begin{example}
We have that 
$$\mu_{\{1,2\},\{3,4,5\}}(\phm_{12} \otimes \phm_{35/4})=\phm_{12/35/4}$$
and 
$$\Delta_{\{1,2\},\{3,4,5\}}(\phm_{12/35/4})=\phm_{12}\otimes \phm_{35/4}.$$
\end{example}

\subsection{The Fock functor}
We have a category $\Sp$ where its objects are vector species, and its morphisms are natural transformations. The \emph{Fock functor} is 
$$\calK: \Sp \rightarrow \Vect_\bbQ$$ defined by 
$$\calK(\bP)=\bigoplus_{n\geq 0} \bP([n]).$$
When $\bP$ is a Hopf monoid, $\calK(\bP)$ is a Hopf algebra with the following product and coproduct. Let $v\in \bP([n])$ and $w\in \bP([m])$. We have an order preserving bijection ${\rm shift}$ between $[m]$ and $[n+1,m+n]=\{n+1,n+2,\dots, n+m\}$. 
The product of $\calK(\bP)$ is defined as follows,
\begin{equation}\label{Eq:Pprod}
\begin{array}{cccc}
\mu: & \calK(\bP)\otimes  \calK(\bP)&  \rightarrow & \calK(\bP)\\
& v\otimes w & \mapsto & \mu_{[n],[n+1,m+n]}(v\otimes \bP[{\rm shift}](w)).
\end{array} 
\end{equation} 
  The coproduct of $\calK(\bP)$ is defined as follows,
\begin{equation}\label{Eq:Pcoprod}
\scriptstyle{
\begin{array}{cccc}
\Delta: & \calK(\bP)&  \rightarrow & \calK(\bP)\otimes  \calK(\bP)\\
& v & \mapsto & \sum_{[n]=S_1\sqcup S_2}\bP[\St]\otimes \bP[\St](\Delta_{S_1,S_2}(v)).
\end{array} 
}
\end{equation}

As an example, consider $\calK(\bPi)$. We have 
$$\calK(\bPi)  =\calK(\bPi)^0 \oplus \calK(\bPi)^1 \oplus \cdots$$ where $\calK(\bPi)^0 = \spam \{1\}$ and {the} $n$th graded piece for $n\geq 1$ has the following bases, which respectively are called the $n$th graded piece of the \emph{$\mhm$-, $\phm$-, and $\xhm$-basis of $\calK(\bPi)$},
$$\begin{array}{rclclcc}
\calK(\bPi)^n&=& \mathbb{Q}\text{-span}\{ \mhm_\pi \suchthat \pi\vdash [n]\} &=&\mathbb{Q}\text{-span}\{ \phm_\pi \suchthat \pi\vdash [n]\}
&=& \mathbb{Q}\text{-span}\{ \xhm_\pi \suchthat \pi\vdash [n]\}.
\end{array}$$

Given two set partitions $\sigma\vdash [n]$ and $\tau\vdash [m]$, we have
\begin{equation}\label{eq:SetProduct}
\begin{array}{cccc}
\mu: & \calK(\bPi)\otimes  \calK(\bPi)&  \rightarrow & \calK(\bPi)\\
& \phm_\sigma\otimes \phm_\tau & \mapsto & \mu_{[n],[n+1,m+n]}(\phm_\sigma \otimes \bPi[{\rm shift}]( \phm_\tau))=\phm_{\sigma|\tau}.
\end{array} 
\end{equation}
Also, for $\pi\vdash [n]$, 
\begin{equation}\label{eq:SetCoproduct}
\scriptstyle{
\begin{array}{cccc}
\Delta: & \calK(\svw{\bPi})&  \rightarrow & \calK(\svw{\bPi})\otimes  \calK(\svw{\bPi})\\
& \phm_\pi & \mapsto & \sum_{[n]= S_1\sqcup S_2}\bP[\St]\otimes \bP[\St](\Delta_{S_1,S_2}(\phm_\pi))=\sum_{[n]=S_1\sqcup S_2:\atop {\pi=\pi|_{S_1}\sqcup \pi|_{S_2}}}\phm_{\St(\pi|_{S_1})}\otimes \phm_{\St(\pi|_{S_2})}.
\end{array} 
}
\end{equation}
Therefore, by Equations (\ref{eq:prodNCSym}) and $(\ref{eq:coprodNCSym})$, we have that there is an isomorphism from  $ \calK(\bPi)$ to $\NCSym$  where through this isomorphism, for every $\pi\vdash [n]$, $\mhm_\pi$, $\phm_\pi$, and $\xhm_\pi$ in $\calK(\bPi)$ are mapped to  $\mnc_\pi$, $\pnc_\pi$, and $\xnc_\pi$ in $\NCSym$, respectively.

\section{The coproduct formula for \svw{the} $\xnc$-basis of $\NCSym$}\label{sec:coproduct}

%
%
%

In the following theorem, we give our product and coproduct formulas for \svw{the} $\xhm$-basis of the Hopf monoid of set partitions $\bPi$. 
\begin{theorem}\label{thm:prod-coprod}
We have the following. 
\begin{enumerate} 
\item Let $S= S_1\sqcup S_2$, $A\vdash S_1$, and $B\vdash S_2$. Then $$\mu_{S_1,S_2}(\xhm_A\otimes \xhm_B)=\xhm_{A\sqcup B}.$$
\item Let $A\vdash S$. Then 
$$\Delta_{S_1,S_2}(\xhm_A)=\sum_{B\vdash S_1,C\vdash S_2} c_{B,C} \xhm_{B} \otimes \xhm_{C},$$
where 
$$c_{B,C}=\sum_{D\vdash S: \atop{D=D|_{S_1} \sqcup D|_{S_2}\leq A, \atop B \leq D|_{S_1},C\leq D|_{S_2}}} \mymu(D,A).$$
\end{enumerate} 
\end{theorem}

\begin{proof} (1) Using the same argument as 
\cite[Lemma 4.2]{BHRZ}, we have  
\begin{equation*}
\begin{split} 
\mu_{S_1,S_2}(\xhm_A\otimes \xhm_B)&=\sum_{C\leq A} \sum_{D\leq B} \mymu(C,A)\mymu(D,B)\mu_{S_1,S_2}(\svw{\phm_C}\otimes \phm_D)
\\
&=\sum_{C\leq A} \sum_{D\leq B} \mymu(C,A)\mymu(D,B)\svw{\phm_{C\sqcup D}}\\
&=\sum_{E\leq \svw{A\sqcup B}} {\mymu}(E,A\sqcup B) \phm_E=\xhm_{A\sqcup B}.
\end{split} 
\end{equation*}

(2) For \svw{the} coproduct formula, 
we have that 
\begin{equation*}
\begin{split}
\Delta_{S_1,S_2}(\xhm_A)&=\Delta_{S_1,S_2}\left(\sum_{D\leq A} \mymu(D,A) \phm_D\right)
\\
& =\sum_{D\leq A} \mymu(D,A) \Delta_{S_1,S_2}\left(\phm_D\right)\\
&=\sum_{D\leq A:\atop{D=D|_{S_1}\sqcup D|_{S_2}} } \mymu(D,A)\phm_{D|_{S_1}} \otimes \phm_{D|_{S_2}} \\&=\sum_{D\leq A:\atop {D=D|_{S_1}\sqcup D|_{S_2}} } \mymu(D,A)\sum_{{B\leq D|_{S_1}, C\leq D|_{S_2}} }\xhm_{B} \otimes \xhm_{C} 
\\
&=\sum_{D\vdash S: \atop{D=D|_{S_1}\sqcup D|_{S_2}\leq A, \atop {B\leq D|_{S_1}, C\leq D|_{S_2}}}} \mymu(D,A)\xhm_{B} \otimes \xhm_{C},
\end{split} 
\end{equation*} 
where the third equality is by Equation (\ref{eq:cop-pi}).
\end{proof} 

\begin{example}
We have that 
$$\mu_{\{1,2\},\{3,4,5\}}(\xhm_{12} \otimes \xhm_{35/4})=\xhm_{12/35/4}.$$
Let $A=123/4$. We want to compute 
$\Delta_{\{1,2\},\{3,4\}}(\xhm_{A}).$ 

For $B=12\vdash \{1,2\}$ and $C=3/4\vdash \{3,4\}$, the set partition $D=12/3/4$ satisfies $D=D|_{\{1,2\}}\sqcup D|_{\{3,4\}}\leq A$ and $B\leq D|_{\{1,2\}}, C\leq D|_{\{3,4\}}$. Also, $\mymu(D,A)=-1$. Thus, $c_{B,C}=-1$.

For $B=1/2\vdash \{1,2\}$ and $C=3/4\vdash \{3,4\}$, the set partitions $D=12/3/4, 1/2/3/4$ satisfy $D=D|_{\{1,2\}}\sqcup D|_{\{3,4\}}\leq A$ and $B\leq D|_{\{1,2\}}, C\leq D|_{\{3,4\}}$. Also, $\mymu(12/3/4,A)=-1$ and $\mymu(1/2/3/4,A)=2$. Thus, $c_{B,C}=1$.

For any other $(B,C)$ with $B\vdash \{1,2\}$ and $C\vdash \{3,4\}$, we have $c_{B,C}=0$. 

Therefore, 
$$\Delta_{\{1,2\},\{3,4\}}(\xhm_{123/4})=-\xhm_{12}\otimes \xhm_{3/4}+\xhm_{1/2}\otimes \xhm_{3/4}.$$
\end{example} 


\svw{We now apply Theorem~\ref{thm:prod-coprod} to obtain a coproduct formula for the extra basis in $\NCSym$. Note that  to obtain this, we only need to establish a coproduct formula for $x_{\stn}$. This is because, firstly, for set partitions $\pi$ and $\sigma$,  we have by \cite[Theorem 1.3]{BHRZ} that
\begin{equation}\label{eq:slashp}
x_\pi x_\sigma = x_{\pi\slashp \sigma}. \end{equation}Secondly, for any set partition $\pi\vdash [n]$, there is a permutation $\eta\in \mathfrak{S}_n$ such that $\eta(\llbracket \lambda \rrbracket)=\pi$, and so by \cite[Equation (4.2)]{DvW}, $\eta\circ x_{\llbracket \lambda\rrbracket}=x_{\eta(\llbracket \lambda\rrbracket)}=x_\pi$. Therefore, we can use these two facts together with the corollary below to compute a coproduct formula for the value of the coproduct of $\NCSym$ at $x_{\pi}$ for any set partition $\pi$. 
}

\begin{corollary}\label{cor:coproduct}
In $\NCSym$, the coefficient of $x_\sigma\otimes x_\tau$ in $\Delta(x_{\stn})$ is 
$$ \sum_{ \nu\vdash [n], [n]=S_1\sqcup S_2: \atop { \nu= \nu|_{S_1}\sqcup  \nu|_{S_2}, \atop  \sigma\leq \St( \nu|_{S_1}), \tau\leq \St( \nu|_{S_2}) }}  (-1)^{\ell( \nu)-1} (\ell( \nu)-1)!.$$
\end{corollary} 

\begin{proof} We have seen that there is an isomorphism from $\calK(\bPi)$  to $\NCSym$, so that through this isomorphism $\xhm_\pi$ in $\calK(\bPi)$ maps to $x_\pi$ in $\NCSym$. Thus, in $\NCSym$, the coefficient of $x_\sigma\otimes x_\tau$ in $\Delta(x_{\stn})$ is equal to the  coefficient of $\xhm_\sigma\otimes \xhm_\tau$ in $\Delta(\xhm_{\stn})$, the value of the coproduct of $\calK(\bPi)$ at $\xhm_{\stn}$.
We have by Equation \eqref{Eq:Pcoprod},
\begin{equation*}
\begin{split} 
 \Delta(\xhm_{\stn})&= \sum_{[n]=S_1\sqcup S_2}{\bf \Pi}[\St]\otimes {\bf \Pi}[\St](\Delta_{S_1,S_2}(\xhm_{\stn})) 
 \\
 &=\sum_{[n]=S_1\sqcup S_2}{\bf \Pi}[\St]\otimes {\bf \Pi}[\St] \left( \sum_{\nu=\nu|_{S_1}\sqcup \nu|_{S_2} \atop {\sigma\leq \nu|_{S_1}, \tau\leq \nu|_{S_2}}} \mymu(\nu,\stn)\xhm_{\sigma} \otimes \xhm_{\tau}\right)
 \\
 & =\left( \sum_{ \nu\vdash [n],[n] = S_1\sqcup S_2:  \atop  \sigma \leq \nu|_{S_1}, \tau\leq \nu|_{S_2}}  (-1)^{\ell( \nu)-1} (\ell( \nu)-1)! \right) \bx_{\St(\sigma)} \otimes \bx_{\St(\tau)}
 \\
& =\left( \sum_{ \nu\vdash [n],[n] = S_1\sqcup S_2:  \atop  \sigma \leq \St( \nu|_{S_1}), \tau\leq \St(\nu|_{S_2}) }  (-1)^{\ell( \nu)-1} (\ell( \nu)-1)! \right) \bx_\sigma \otimes \bx_\tau,
 \end{split} 
\end{equation*}
where the second equality is by Theorem \ref{thm:prod-coprod}, \svw{and the third equality is by Equation~\eqref{eq:mymun}.}
\end{proof} 
%
%
%
%
%

\begin{example} 
Let $\sigma=1/2$ and $\tau=12$. We want to compute the coefficient of $x_{1/2}\otimes x_{12}$ in $\Delta(x_{\llbracket 4\rrbracket})$. All tuples $(\nu,S_1,S_2)$ that satisfy 
$$ \nu\vdash [4], [4]= S_1\sqcup S_2, 1/2 \leq \St(\nu|_{S_1}), 12\leq \St(\nu|_{S_2})$$ are  
$$(1/2/34, \{1,2\},\{3,4\}), (12/34, \{1,2\},\{3,4\}), (1/3/24, \{1,3\},\{2,4\}), (13/24, \{1,3\},\{2,4\}),$$ 
$$ (1/4/23, \{1,4\},\{2,3\}), (14/23, \{1,4\},\{2,3\}), (2/3/14, \{2,3\},\{1,4\}), (23/14, \{2,3\},\{1,4\}),$$
$$(2/4/13, \{2,4\},\{1,3\}), (24/13, \{2,4\},\{1,3\}),(3/4/12, \{3,4\},\{1,2\}), \svw{(34/12, \{3,4\},\{1,2\}).}$$
Therefore, the coefficient of $x_{1/2}\otimes x_{12}$ in $\Delta(x_{\llbracket 4 \rrbracket})$ is 
$$
6 \times (2+(-1))=6. 
$$
\end{example}

We conclude this section with the following proposition, 
\svw{after a definition. Let the \emph{extra symmetric function in $\Sym$}, $x_\lambda$, be given by
\begin{equation}\label{eq:xinsym}x_\lambda=x_{\lambda(\llbracket \lambda \rrbracket)}=\rho(x_{\llbracket \lambda \rrbracket}).\end{equation}}

\begin{proposition}\label{prop:xbasisSym}
The set $\{x_{\lambda}: \lambda \vdash n\}$ is a basis for $\Sym^n$. Moreover, if \svw{$\lambda$ and $\gamma$ are partitions,} then 
$$\svw{x_\lambda  x_\gamma=x_{\lambda|\gamma}.}$$
\end{proposition} 

\begin{proof}
The set $\{x_{\lambda}: \lambda \vdash n\}$ is a basis for $\Sym^n$ \svw{because} $\mymu(\llbracket \lambda \rrbracket,\llbracket \lambda \rrbracket)=1$ and \svw{by definition and Lemma~\ref{lem:RSrho}}
$$x_\lambda=\rho(x_{\llbracket \lambda \rrbracket})=\sum_{\sigma\leq  \llbracket \lambda \rrbracket} \mymu(\sigma, \llbracket \lambda \rrbracket) p_{\lambda(\sigma)}.$$ Also, if \svw{$\lambda$ and $\gamma$ are partitions,}  \svw{then by Equation~\eqref{eq:slashp}} and the fact that $\rho$ is an algebra morphism we have that $$\svw{x_\lambda  x_\gamma=\rho(x_{\llbracket \lambda \rrbracket})\rho(x_{\llbracket \gamma \rrbracket})= \rho(x_{\llbracket \lambda \rrbracket}x_{\llbracket \gamma \rrbracket})= \rho(x_{\llbracket \lambda \rrbracket \slashp\llbracket \gamma \rrbracket})= x_{\lambda(\llbracket \lambda \rrbracket \slashp\llbracket \gamma \rrbracket)}= x_{\lambda|\gamma}.}$$\end{proof} 
\svw{\begin{example}\label{ex:sbasisSym}
$x_{3211}x_{221}= x_{3222211}$
\end{example}}
\section{\svw{Changes} of bases}\label{sec:change} 
Given any basis $\{b\}$  of a Hopf algebra $H$, we say an element $h$ in $H$ is \svw{\emph{$b$-positive}} (\svw{respectively} \svw{\emph{$b$-negative}}) if all coefficients appearing in the expansion of $h$ in terms of the basis $\{b\}$ are positive (\svw{respectively} negative).
In this section, we describe the expansion of the $x$-basis elements in terms of \svw{the} $m$-, $p$-, and $e$-basis of $\NCSym$, and we investigate if these expansions are positive or negative. 
In the end, we 
prove that $x_{\stn}$ is $e$-positive in $\NCSym$ if and only if $x_n$ is $e$-positive in $\Sym$.

In order to write  $\xnc_\pi$ in terms of \svw{the} $\mnc$-basis, we need the following definitions. \svw{Given a graph $G$,
 a} \emph{stable partition} of $G$ is a partition of the vertices of $G$ such that the vertices in each block are independent. Let $S(G)$ be the set of all stable partitions of $G$.  Note that the \emph{chromatic polynomial} $\chi_G(k)$, which is the number of proper colourings of the graph $G$ with $k$ colours,  can be computed as 
\begin{equation}\label{eq:chi}\chi_G(k)=\sum_{\pi\in S(G)} k(k-1)\cdots (k-\ell(\pi)+1).\end{equation}
For any set partition $\sigma$ of $[n]$, define $K_\sigma$ to be the complete multipartite graph on vertices $[n]$ with an edge between $i$ and $j$ if and only if they lie in distinct blocks of $\sigma$.

\begin{theorem}\label{thm:x-to-m}
Let $n\geq 1$. Then 
$$\xnc_{\stn}=(-1)^{n-1}\sum_{\sigma\vdash [n]} c_\sigma \mnc_\sigma$$ where $c_\sigma$ is the number of acyclic orientations of $K_\sigma$ with a unique sink at any fixed vertex. 
\end{theorem}

\begin{proof}
We have \svw{by \cite[Theorem 3.1]{RS}} that 
$$\pnc_\tau=\sum_{\sigma\geq \tau} \mnc_\sigma.$$
Therefore, \svw{by definition}
$$\xnc_\pi=\sum_{\tau\leq \pi} \mymu(\tau,\pi) \pnc_\tau=\sum_{\tau\leq \pi} \mymu(\tau,\pi) \sum_{\sigma\geq \tau} \mnc_\sigma=\sum_{\sigma\vdash [n]} \left(\sum_{\tau\leq \pi\wedge \sigma}\mymu(\tau,\pi) \right) \mnc_\sigma.$$
And so \svw{by Equation~\eqref{eq:mymun}}
$$\xnc_{\stn}=\sum_{\sigma\vdash [n]} \left(\sum_{\tau\leq \sigma}\mymu(\tau,\stn) \right) \mnc_\sigma=\sum_{\sigma\vdash [n]} \left(\sum_{\tau\leq \sigma}(-1)^{\ell(\tau)-1}(\ell(\tau)-1)! \right) \mnc_\sigma.$$
Thus, for any $\sigma\vdash [n]$, the coefficient of $\mnc_\sigma$ in $\xnc_{\stn}$ is 
\begin{equation*}
\begin{split}
\sum_{\tau\leq \sigma}(-1)^{\ell(\tau)-1}(\ell(\tau)-1)!&=\sum_{\tau\in S(K_\sigma)} (-1)^{\ell(\tau)-1}(\ell(\tau)-1)!\\
&= \left(\sum_{\tau\in S(K_\sigma)} (0-1) (0-2)\cdots (0-(\svw{\ell(\tau)}-1)) \right)\\
&= \left( \frac{1}{k} \chi_{K_\sigma} \right)(0),
\end{split}
\end{equation*}\svw{by Equation~\eqref{eq:chi}.}
By \cite[Theorem 8]{EM}, this is the number of acyclic orientations of $K_\sigma$ with
a unique sink at any fixed vertex, multiplied by $(-1)^{n-1}$.
\end{proof}

\begin{example}\label{ex:xasm}
We have that $$\xnc_{\llbracket 3 \rrbracket}=2\mnc_{1/2/3} + \mnc_{1/23} + \mnc_{12/3} + \mnc_{13/2}.$$
\end{example} 

\svw{Now we turn to the $p$-basis in $\NSym$.}
\begin{proposition}\label{prop:xasp} We have the following.
\begin{enumerate} 
\item $\xnc_{\stn}=\sum_{\sigma\vdash[n]}(-1)^{\ell(\sigma)-1}(\ell(\sigma)-1)! \pnc_\sigma.$
\item $\omega(\xnc_{\stn})=(-1)^{n-1} \sum_{\sigma\vdash[n]}(\ell(\sigma)-1)! \pnc_\sigma$.
\item 
$\omega(\xnc_{\pi})$ is either $p$-positive or $p$-negative.
\end{enumerate} 
\end{proposition} 

\begin{proof}
Part (1) follows from Equations \svw{(\ref{eq:mymun})} and (\ref{eq:p-basis}). Part (2) follows from part (1) and the fact that $\omega(\pnc_{\sigma})=(-1)^{n-\ell(\sigma)}\pnc_{\sigma}.$ Given the set partition $\pi$ with $\lambda(\pi)=\lambda=\lambda_1\lambda_2 \svw{\cdots} \lambda_l$, \svw{we know} there is a permutation $\eta$ such that $\eta(\llbracket \lambda \rrbracket)=\pi$. By the definition of the involution $\omega$ in Equation (\ref{eq:omega}), we have $\eta(\omega(\xnc_{\llbracket \lambda \rrbracket}))=\omega(\xnc_{\eta(\llbracket \lambda \rrbracket)})=\omega(\xnc_\pi)$. Therefore, $\omega(\xnc_\pi)$ is $p$-positive (\svw{respectively} $\pnc$-negative) if and only if $\omega(\xnc_{\llbracket \lambda \rrbracket})$ is $p$-positive (\svw{respectively} $\pnc$-negative). \svw{However,} from part (2) and the fact that 
$$\omega(\xnc_{\llbracket \lambda \rrbracket})=\omega( x_{\llbracket \lambda_1\rrbracket})\omega(x_{\llbracket \lambda_2\rrbracket })\cdots\omega(x_{\llbracket \lambda_l\rrbracket}),$$ we conclude that $\omega(\xnc_{\llbracket \lambda \rrbracket})$ is either $p$-positive or $p$-negative.
\end{proof}

\svw{Finally, we} write the expansion of $\xnc_\pi$ in terms of \svw{the} $e$-basis of $\NCSym$, \svw{which leads us to a conjecture.} Let $\hat{0}=1/2/\cdots/n$. We have by \cite[Theorem 3.4]{RS} that 
$$
p_\tau=\frac{1}{\mymu(\hat{0},\tau)}\sum_{\sigma\leq \tau} u(\sigma,\tau)e_\sigma.
$$
Therefore, \svw{by definition}
\begin{equation*}
\begin{split}
x_\pi=\sum_{\tau\leq \pi} \mymu(\tau,\pi) p_\tau&=\sum_{\tau\leq \pi} \mymu(\tau,\pi)\frac{1}{\mymu(\hat{0},\tau)}\sum_{\sigma\leq \tau} \mymu(\sigma,\tau)e_{\sigma}\\
&=\sum_{\sigma\vdash [n]}\left( \sum_{\sigma\leq \tau \leq \pi}\frac{\mymu(\tau,\pi)\mymu(\sigma,\tau)}{\mymu(\hat{0},\tau)}  \right) \svw{e_\sigma.} \\
\end{split}
\end{equation*}

%
%
%

%
%
%

\begin{conjecture}\label{conj}
Let $\sigma\vdash [n]$. \svw{Then} the coefficient of $e_\sigma$ in the expansion of $x_\stn$ in terms of \svw{the} $e$-basis \svw{is positive if $n-1$ is even, and negative if $n-1$ is odd.} \end{conjecture}

We conclude \svw{our} paper with the following theorem. 

\begin{theorem}\label{thm:e-pos} We have that $x_{\stn}$ is $\enc$-positive in $\NCSym$ if and only if $x_{n}$ is $e$-positive in $\Sym$.
\end{theorem} 
\begin{proof} If $\xnc_{\stn}$ is $e$-positive in $\NCSym$, clearly $\xnc_{n}$ is $e$-positive in $\Sym$ \svw{by $\rho$.} 

Now assume that $\xnc_{n}$ is $e$-positive in $\Sym$. 
By \cite[Equation (4)]{RS}, the number of set partitions with shape $\lambda$ is $\frac{n!}{\lambda!\lambda^{!}}.$
Thus, for  $\sigma\vdash [n]$ with $\lambda(\sigma)=\lambda$,
 $$R(p_\sigma)=\frac{1}{n!}\sum_{\eta\in\mathfrak{S}_n} p_{\eta(\sigma)}=\frac{\lambda! \lambda^{!}}{n!}\sum_{\lambda(\tau)=\lambda} p_{\tau}.$$ Consider that in $x_{\stn}$ the coefficients of $p_\tau$ with $\lambda(\tau)=\lambda$ are the same \svw{by Proposition~\ref{prop:xasp}.} Therefore, $x_{\stn}$ is in the image of $R$, and so $x_{\stn}=R\rho(x_{\stn})=R(x
 _{n})$.

 
  Moreover, by Lemma  \ref{lem:RSrho},
 $$\lambda!e_\lambda=\rho \left(\frac{1}{n!}\sum_{\eta\in \mathfrak{S}_n} e_{\eta(\sigma)}\right).$$ Since $\frac{1}{n!}\sum_{\eta\in \mathfrak{S}_n} e_{\eta(\sigma)}$ is in the image of $R$, we have 
 $$R(e_\lambda)=R\rho \left(\frac{1}{\lambda! n!}\sum_{\eta\in \mathfrak{S}_n} e_{\eta(\sigma)}\right)=\frac{1}{\lambda! n!}\sum_{\eta\in \mathfrak{S}_n} e_{\eta(\sigma)}.$$

 Since $x_n$ is $e$-positive in $\Sym$ \svw{by assumption, we have that}  $x_{\stn}=R(x_n)$ is a positive linear combination of the elements in 
$$\{ R(e_\lambda): \lambda\vdash n\}.$$ Thus, $x_{\stn}$ is a positive linear combination of the elements in  $$\left\{ \frac{1}{\lambda!n!}\sum_{\eta\in \mathfrak{S}_n} e_{\eta(\sigma)}: \sigma\vdash [n] \right\}.$$
Therefore, $x_{\stn}$ is $\enc$-positive in $\NCSym$.
%
\end{proof}

\section{Acknowledgments}\label{sec:ack}  The authors would like to thank Victor Wang for helpful conversations and for sharing Theorem~\ref{thm:x-to-m} with us.


\bibliographystyle{plain}

\def\cprime{$'$}

\end{document}